\documentclass[sn-mathphys-num]{sn-jnl}

\usepackage{graphicx}%
\usepackage{multirow}%
\usepackage{amsmath,amssymb,amsfonts}%
\usepackage{amsthm}%
\usepackage{mathrsfs}%
\usepackage[title]{appendix}%
\usepackage{xcolor}%
\usepackage{textcomp}%
\usepackage{manyfoot}%
\usepackage{booktabs}%
\usepackage{algorithm}%
\usepackage{algorithmicx}%
\usepackage[noend]{algpseudocode}%
\usepackage{listings}%

\newtheorem{theorem}{Theorem}
\newtheorem{proposition}[theorem]{Proposition}%

\theoremstyle{thmstyletwo}%
\newtheorem{remark}{Remark}%

\theoremstyle{thmstylethree}%
\newtheorem{definition}{Definition}%

\usepackage{lipsum}
\usepackage{amsfonts}
\usepackage{graphicx}
\usepackage{epstopdf}
\ifpdf
  \DeclareGraphicsExtensions{.eps,.pdf,.png,.jpg}
\else
  \DeclareGraphicsExtensions{.eps}
\fi

\usepackage{amsmath}
\usepackage{amssymb}
\usepackage{hyperref}
\usepackage{soul}
\def\eps{\epsilon}

\usepackage[shortlabels]{enumitem}
\usepackage{resizegather}
\usepackage{cleveref}

\renewcommand{\Re}{\mathbb{R}}

\crefname{assumption}{Assumption}{Assumptions}

\begin{document}

\title[A globally convergent gradient method with momentum]{A globally convergent gradient method with momentum}


\author*[1]{\fnm{Matteo} \sur{Lapucci}}\email{matteo.lapucci@unifi.it}

\author[2]{\fnm{Giampaolo} \sur{Liuzzi}}\email{liuzzi@diag.uniroma1.it}

\author[2]{\fnm{Stefano} \sur{Lucidi}}\email{lucidi@diag.uniroma1.it}

\author[1]{\fnm{Davide} \sur{Pucci}}\email{davide.pucci@unifi.it}

\author[2]{\fnm{Marco} \sur{Sciandrone}}\email{sciandrone@diag.uniroma1.it}

\affil*[1]{\orgdiv{Dipartimento di Ingegneria dell'Informazione Firenze}, \orgname{Università di Firenze}, \orgaddress{\street{Via di Santa Marta 3}, \city{Firenze}, \postcode{50135}, \state{FI}, \country{Italy}}}

\affil[2]{\orgdiv{Dipartimento di Ingegneria Informatica, Automatica e Gestionale ``Antonio Ruberti''}, \orgname{Sapienza Università di Roma}, \orgaddress{\street{Via Ariosto 25}, \city{Roma}, \postcode{00185}, \state{RM}, \country{Italy}}}


\abstract{In this work, we consider smooth unconstrained optimization problems and we deal with the class of gradient methods with momentum, i.e., descent algorithms where the search direction is defined as a linear combination of the current gradient and the preceding search direction. This family of algorithms includes nonlinear conjugate gradient methods and Polyak's heavy-ball approach, and is thus of high practical and theoretical interest in large-scale nonlinear optimization. 
We propose a general framework where the scalars of the linear combination defining the search direction are computed simultaneously by minimizing the approximate quadratic model in the 2 dimensional subspace.
This strategy allows us to define a class of gradient methods with momentum enjoying global convergence guarantees and an optimal worst-case complexity bound in the nonconvex setting. Differently than all related works in the literature, the convergence conditions are stated in terms of the Hessian matrix of the bi-dimensional quadratic model. To the best of our knowledge, these results are novel to the literature. 
Moreover, extensive computational experiments show that the gradient method with momentum here presented is competitive with respect to other popular solvers for nonconvex unconstrained problems.
}



\maketitle

\section{Introduction}
In this work we consider unconstrained optimization problems
\begin{equation*}\label{mainprob}
	\min_{x\in\mathbb{R}^n} f(x),
\end{equation*}
where $f:\Re^n\to\Re$ is a continuously differentiable objective function. We do not assume that the function is convex.
We focus on first order descent methods that exploit information from the preceding iteration to determine the search direction and the stepsize at the current one. We will hence refer to \textit{gradient methods with momentum}, i.e., to algorithms defined by an iteration of the generic form
\begin{equation}\label{framework}
x_{k+1} = x_k - \alpha_k\nabla f(x_k) + \beta_k(x_k-x_{k-1}),
\end{equation}
where  $\alpha_k>0$ is the stepsize, and $\beta_k>0$ is the momentum weight.
Partially repeating the previous step has the effect of controlling oscillation and providing acceleration in low curvature regions. All of this can, in principle, be achieved only exploiting already available information: no additional function evaluations are required to be carried out. This feature makes the addition of momentum terms appealing in large-scale settings and, in particular, in the deep learning context \cite{bottou2018optimization,wright2022optimization}.

The best-known and most important gradient methods
with momentum arguably are:
\begin{itemize}
\item[-] Polyak's heavy-ball method \cite{polyak1964some,polyak1987introduction};
\item [-] conjugate gradient methods (see, e.g., \cite{grippo2023introduction}).
\end{itemize}

Conjugate gradient methods can be described by the updates
\begin{gather*}
	x_{k+1} = x_k+\alpha_k d_k,\qquad
	d_{k+1} = -\nabla f(x_{k+1}) + \beta_{k+1}d_k,
\end{gather*}
where $\alpha_k$ is computed by means of a line search,
whereas $\beta_k$ is obtained according to one of many rules from the literature (see, e.g., \cite{grippo2023introduction}).
The update of conjugate gradient methods can be rewritten as 
\begin{align*}
	x_{k+1}&=  x_k+\alpha_k(-\nabla f(x_k)+\beta_kd_{k-1})\\&=x_k+\alpha_k\left(-\nabla f(x_k)+\frac{\beta_k}{\alpha_{k-1}}(x_k-x_{k-1})\right)\\&=
	x_k - \alpha_k \nabla f(x_k) + \hat{\beta}_k(x_k-x_{k-1}),
\end{align*}
so that it can be viewed as a gradient method with momentum according to definition \eqref{framework}.

The convergence theory of nonlinear conjugate gradient methods has been a research topic for about 30 years and now it can be considered well-established, while some complexity results have been stated only recently \cite{chan2022nonlinear,neumaier2024globally}. Several proposed conjugate gradient methods can be considered sound and efficient tools for unconstrained optimization.
Recently, a new class of conjugate gradient methods, known as subspace minimization conjugate gradient (SMCG) methods, have been proposed in the literature \cite{liu2024new,sun2021class,yang2017subspace,yuan1995subspace}. We will discuss more in detail this class of approaches later in this work, since they are related to the framework proposed here.

On the other hand, the heavy-ball algorithm is directly described by an iteration of the form \eqref{framework},
where $\alpha_k$ and $\beta_k$ typically are fixed positive values\cite{sra2012optimization}; in principle, suitable constants should be chosen depending on properties of the objective function (e.g., using Lipschitz constant of the gradient or the constant of strong convexity) \cite{ghadimi2015global,lessard2016analysis,polyak1987introduction}. In practice, however, this information is often not accessible and thus $\alpha_k$ can be chosen by a line search while the momentum parameter $\beta_k$ is usually blindly set to some (more or less) reasonable value. 
Convergence results for the heavy-ball method  have been proven in the convex case \cite{ghadimi2015global,polyak1964some,saab2022adaptive}, while
the convergence of the method  in the nonconvex case is still an open problem.

Thus, algorithmic issues related to the choice of the two
parameters $\alpha_k$ and $\beta_k$ and the theoretical gap related
to the convergence in the nonconvex case did not allow, until now, to include the heavy-ball method within the class of sound methods for smooth unconstrained optimization. As a matter of fact,
there does not exist any popular software implementation of the method. 
However, heavy-ball type momentum terms are consistently and effectively used within modern frameworks of stochastic optimization for neural network training \cite{liu2020improved,sutskever2013importance}.

Both conjugate gradient and heavy-ball methods are therefore of practical interest in large-scale nonlinear optimization settings. We hence believe it is worth to focus on the study of the general class of gradient methods with momentum, in order to possibly define convergent algorithms improving the efficiency of standard nonlinear conjugate gradient methods.

To this aim, we draw inspiration from the idea presented in \cite{yuan1995subspace}, where 
the search direction is computed by minimizing the approximate quadratic model in the 2 dimensional
subspace spanned by the current gradient and the last search direction.
According to this approach,
the scalars $\alpha_k$ and $\beta_k$ are not prefixed, but rather they are simultaneously determined by a bidimensional search.
We define a general framework of gradient methods with momentum
based on a simple Armijo-type line search, and we prove global convergence results under the first-order smoothness assumptions only. 
We also derive specific algorithms with momentum from the general framework.
Furthermore, we provide complexity results, proving for the proposed gradient method with momentum the worst-case complexity bound of  ${\cal O}(\epsilon^{-2})$, which is optimal for first order algorithms in the nonconvex setting \cite{carmon2020lower}.

Up to our knowledge, the theoretical results presented here are novel in the literature. Known convergence and complexity properties for similar algorithmic frameworks are indeed based on assumptions that are stronger than ours or, as an alternative, can only be enforced in practice resorting to restarting techniques. Our analysis here does not make hypotheses concerning second order information for the objective nor linear independence between gradient and momentum directions; in addition, the conditions resulting from our analysis can be rigorously handled directly in the subspace, which is a crucial aspect from the computational perspective.

Extensive numerical experiments show that the gradient methods with momentum here presented outperform popular implementations of the conjugate gradient method and a widespread limited-memory Quasi-Newton method like
L-BFGS algorithm \cite{liu1989limited}. Moreover, while the state-of-the-art conjugate gradient based procedure implemented in \texttt{cg\_descent} software \cite{hager2006algorithm,hager2013limited} proves to be clearly superior, the proposed method may  represent an interesting baseline for devising a solver competitive with the state of the art.

The rest of the paper is organized as follows: we describe the main idea of the work in \cref{sec:main_idea}, with a focus on some important related works in \cref{sec:related}; in \cref{sec:existence} we discuss conditions for the proposed method to be well defined. Then, we discuss in \cref{sec:grad_rel_dir} the issue of guaranteeing that the employed search directions are gradient related; in particular, we consider the cases where the property descends from assumptions on either $n\times n$ or $2\times 2$ matrices (\cref{sec:nn} and \cref{sec:22} respectively). In \cref{sec:alg}, we finally describe the proposed algorithmic framework for gradient methods with momentum, formalizing convergence and complexity results. In \cref{sec:comp_H}, we then discuss concrete strategies to estimate the matrices that are at the core of the proposed method. In \cref{sec:exp}, we report the results of thorough computational experiments empirically showing the potential of the proposed class of approaches. We finally give some concluding remarks in \cref{sec:conclusions}.

\section{The main idea}
\label{sec:main_idea}
A vast class of iterative algorithms for nonlinear unconstrained optimization (namely linesearch algorithms) can be written in a general form as
$$x_{k+1} = x_k+\alpha_kd_k,$$
where $d_k\in \Re^n$ is the search direction and $\alpha_k >0$ is the stepsize. The most typical rules for the choice of the direction follow a general scheme, which is basically given by the following optimization subproblem:
\begin{equation}
	\label{eq:gen_dir_prob}
	\min_{d\in\mathbb{R}^n}\;\nabla f(x_k)^T  d + \frac{1}{2}d^T  B_kd,
\end{equation}
where $B_k$ is a suitable symmetric positive definite matrix.
By properly choosing $B_k$, we retrieve standard methods, in particular:
\begin{itemize}
	\item if $B_k = I$, then we obtain the steepest descent direction $-\nabla f(x_k)$;
	\item if $B_k = \nabla^2 f(x_k)$, then we obtain Newton's direction $-\left[\nabla^2 f(x_k)\right]^{-1}\nabla f(x_k)$;
	\item if $B_k$ is a positive definite matrix obtained with suitable update rules, we obtain standard Quasi-Newton updates.
\end{itemize}
Gradient methods with momentum can be considered in the above framework by adding into \eqref{eq:gen_dir_prob} a suitable constraint on $d$, i.e., 
\begin{equation}
	\label{eq:constrained_dir_prob}
	\begin{aligned}
		\min_{d,\alpha,\beta}\;&  g_k^T  d + \frac{1}{2}d^T  B_kd\\\text{s.t. }& d = -\alpha g_k + \beta s_k,
	\end{aligned}
\end{equation}
where $g_k=\nabla f(x_k)$ and $s_k=x_k-x_{k-1}$.
Then, substituting the constraint into the objective function of \eqref{eq:constrained_dir_prob}, the problem reduces to
\begin{equation}\label{phialfabeta}
\min_{\alpha, \beta} \phi(\alpha,\beta)
\end{equation}
where
\begin{gather*}
	\phi(\alpha, \beta)=-\alpha\| g_k\|^2+\beta g_k^T s_k+\frac{1}{2}\alpha^2 
	g_k^TB_kg_k+ \frac{1}{2}\beta^2 
	s_k^TB_ks_k-\alpha \beta g_k^TB_ks_k,
\end{gather*}
or equivalently
\begin{equation}\label{phialfabeta2}
\phi(\alpha,\beta) = \begin{bmatrix}
	-\| g_k\|^2\\
	g_k^T s_k
\end{bmatrix}^T\begin{bmatrix}
	\alpha\\
	\beta
\end{bmatrix} + \frac{1}{2}\begin{bmatrix}
	\alpha\\
	\beta
\end{bmatrix}^T \begin{bmatrix}
	g_k^TB_kg_k& -g_k^TB_ks_k\\ 
	-g_k^TB_ks_k& s_k^TB_ks_k
\end{bmatrix}\begin{bmatrix}
	\alpha\\
	\beta
\end{bmatrix}. 
\end{equation}
We can denote the $2\times 2$ matrix in the above equation as
\begin{equation}
	\label{eq:Hk_explicit}
	H_k = \begin{bmatrix}
		H_{11} & H_{12}  \\
		H_{12} & H_{22} 
	\end{bmatrix}=\begin{bmatrix}
		g_k^TB_kg_k& -g_k^TB_ks_k\\ 
		-g_k^TB_ks_k& s_k^TB_ks_k
	\end{bmatrix} = P_k^TB_kP_k,
\end{equation}
where we omit the dependence of $H_{ij}$ on $k$ for the sake of notation simplicity and we set $P_k = \left[-g_k\;\; s_k\right]$. Letting $u=[\alpha\;\;\beta]^T$, we can express the problem as
\begin{equation}
	\label{eq:constrained_dir_prob2}
	\min_{u}\;  \frac{1}{2}u^TP_k^T B_kP_ku + g_k^T  P_ku,
\end{equation}
or equivalently
as
\begin{equation}
	\label{eq:constrained_dir_prob2b}
	\min_{u}\;  \frac{1}{2}u^TH_ku + g_k^T  P_ku. 
\end{equation}
Once a solution $u_k=[\alpha_k\;\;\beta_k]^T$ of \eqref{eq:constrained_dir_prob2} is determined, we define
$$
d_k=-\alpha_kg_k+\beta_ks_k
$$
and, provided that $d_k$ is a descent direction, we set
\begin{equation}\label{update}
x_{k+1}=x_k+\eta_kd_k,
\end{equation}
where $\eta_k$ can be determined, for example, by an Armjio-type line search.
\par\medskip\noindent
Now consider
the two-dimensional function  
$$
\psi_k(\alpha,\beta)=f(x_k-\alpha g_k+\beta s_k)
$$
and assume that $f$ is twice continuously differentiable.
We have
\begin{gather*}
	\psi_k(0,0)=f(x_k),
	\qquad
	\nabla \psi_k(0,0)=
	\begin{bmatrix}
		-\|g_k\|^2\\ 
		g_k^Ts_k
	\end{bmatrix},
	\\
	\nabla^2\psi_k(0,0)=
	\begin{bmatrix}
		g_k^T\nabla^2f(x_k)g_k& -g_k^T\nabla^2f(x_k)s_k\\ 
		-g_k^T\nabla^2f(x_k)s_k& s_k^T\nabla^2f(x_k)s_k
	\end{bmatrix}.
\end{gather*}
The quadratic Taylor polynomial for $\psi_k(\alpha,\beta)$ centered at $(0,0)$ is thus
\begin{equation}\label{QuadModel}
\psi_k(\alpha,\beta) = f(x_k)+\begin{bmatrix}
		-\|g_k\|^2\\ 
		g_k^Ts_k
	\end{bmatrix}^T \begin{bmatrix}
			\alpha\\ 
			\beta
		\end{bmatrix}
		+\frac{1}{2}
		\begin{bmatrix}
			\alpha\\ 
			\beta
		\end{bmatrix}^T
		\begin{bmatrix}
	g_k^T\nabla^2f(x_k)g_k& -g_k^T\nabla^2f(x_k)s_k\\ 
	-g_k^T\nabla^2f(x_k)s_k& s_k^T\nabla^2f(x_k)s_k
\end{bmatrix}
\begin{bmatrix}
			\alpha\\ 
			\beta
		\end{bmatrix}.
\end{equation}

The matrix $H_k$ defined by \eqref{eq:Hk_explicit} can then be
viewed as a matrix approximating $\nabla^2\psi_k(0,0)$.
From this point of view,
$H_k$ could in fact be any $2\times 2$ matrix (independent of the $n\times n$ matrix $B_k$) related to the approximation of the quadratic Taylor polynomial of $\psi_k(\alpha,\beta)=f(x_k-\alpha g_k+\beta s_k)$, i.e.,
\begin{equation}\label{QuadModel2}
\psi_k(\alpha,\beta) = f(x_k)+\begin{bmatrix}
		-\|g_k\|^2\\ 
		g_k^Ts_k
	\end{bmatrix}^T \begin{bmatrix}
			\alpha\\ 
			\beta
		\end{bmatrix}
		+\frac{1}{2}
		\begin{bmatrix}
			\alpha\\ 
			\beta
		\end{bmatrix}^T
\begin{bmatrix}
	H_{11} & H_{12}  \\
	H_{12} & H_{22}
\end{bmatrix}
\begin{bmatrix}
			\alpha\\ 
			\beta
		\end{bmatrix}.
\end{equation}
Setting  $u=[\alpha\;\;\beta]^T$, the minimization problem of \eqref{QuadModel2} leads again to problem \eqref{eq:constrained_dir_prob2b}, being $H_k$ a generic $2\times 2$ matrix with suitable properties.

From a theoretical point of view, the following issues must be considered:
\begin{enumerate}[(a)]
\item we shall state conditions ensuring that the two-dimensional subproblem \eqref{eq:constrained_dir_prob2} admits solution;
\item we shall analyze under which conditions the obtained search direction $d_k$, coupled with line search techniques, allow to prove global convergence properties and complexity results for the iterative scheme \eqref{update}.
\end{enumerate}
\par\medskip\noindent
In the following, we will state conditions concerning either the $n\times n$ matrix $B_k$ appearing in \eqref{eq:constrained_dir_prob2} or directly the $2\times 2$ matrix $H_k$ in \eqref{eq:constrained_dir_prob2b} (with no explicit connection to $B_k$), in order to satisfy the above theoretical requirements.

\subsection{Related works and open problems}
\label{sec:related}
Before turning to the theoretical analysis, we shall  briefly discuss some important works directly related to the approach addressed in this paper. 
\begin{description}
	\item[Subspace minimization conjugate gradient (SMCG) methods] The class of SMCG type algorithms,
	for which we refer the reader, for example, to the very recent paper \cite{liu2024new}, consider the two-dimensional subproblem \eqref{phialfabeta}-\eqref{phialfabeta2} by assuming that
	the $n\times n$ matrix $B_k$ satisfies the secant equation
	$$
	B_ks_k=g_k-g_{k-1}.
	$$
	The issue deserving attention becomes that of suitably managing the term $g_k^TB_kg_k=\rho_k$. Several strategies have been proposed in connection with this issue. Some global convergence results are stated, while complexity results are not known. However,	as clearly written in \cite{liu2024new},  {\it``A question is naturally to be asked: can one develop an
		efficient SMCG method without determining the parameter $\rho_k$?''}. In other words, is it possible to define a convergent SMCG method without the need of handling the full matrix $B_k$?  
	By our work, we give a positive answer to this question.
	\smallskip
	\item[Common-directions methods] The framework discussed in \cite{lee2022limited} considers \eqref{eq:constrained_dir_prob2}-\eqref{update} in a more general setting, i.e., by assuming that
	$P_k$ is a general $n \times m_k$ matrix, with $m_k\ge 1$, containing at least a search direction which satisfies an angle condition.  The columns of $P_k$ are required to be linearly independent.
	Convergence results are stated by assuming that $\{B_k\}$ is a sequence of uniformly positive definite symmetric matrices bounded above. 
    However, while the convergence analysis relies on the properties
	of the $n\times n$ matrix $B_k$,  it is crucial, for computational performance, to only deal with the small matrix $H_k$.
  \smallskip
	\item[(Riemannian) dimension-reduced second-order method (RDRSOM)]
	The basic idea of this method, proposed in \cite{zhang2022drsom} and then extended to optimization on Riemannian manifolds \cite{tang2024riemannian}, is similar to that of \cite{lee2022limited}. This time,
	the search direction is determined by solving a cubic-regularized subproblem on the selected subspace which, in the simplest case, is the two-dimensional subspace spanned by the gradient and the previous direction.
The iteration complexity of the proposed algorithm is ${\cal O}(\epsilon^{-3/2})$. However, as remarked by the same authors,
this complexity result is based on a very restrictive assumption requiring
that the Hessian matrix is approximated accurately enough in the considered subspace, which therefore should be adaptively selected. In order to overcome this nontrivial issue, a practical version of the algorithm is presented with iteration complexity ${\cal O}(\epsilon^{-2})$.
Differently from \cite{lee2022limited}, the convergence analysis relies on the properties the $2\times 2$ matrix $H_k$ and  requires that the Hessian matrix of the objective function is Lipschitz continuous. This assumption can not be relaxed since the proposed method is based on the cubic regularization approach.
\end{description}
\par\medskip\noindent
Our main theoretical contribution with respect to all the works on SMCG methods and to \cite{lee2022limited}, thus concerns the definition of conditions on the $2\times 2$ matrix $H_k$
sufficient to ensure that the sequence of search directions $\{d_k\}$
is {\it gradient-related} and hence to state convergence and complexity results for the proposed gradient methods with momentum.  Note that working with only the $2\times 2$ matrix, regardless of the $n\times n$ matrix $B_k$, is fundamental to make these algorithmic frameworks computationally viable at large scale; our result on this issue thus enlightens how it can be done in a theoretically sound manner.

As regards the contribution with respect to \cite{tang2024riemannian}, we do not state assumptions on the Hessian matrix of the objective function:
we only require that the gradient is Lipschitz continuous. Thus,
we attain the same complexity result under much weaker assumptions.


\section{Existence of a solution for the two dimensional problem}
\label{sec:existence}
	In this section we state conditions on the matrices $B_k$ and $H_k$ sufficient to ensure that the subproblem
	\eqref{eq:constrained_dir_prob2b} admits solution.
	The condition on $H_k$ is rather simple to state.
	\begin{proposition}
		\label{prop:existence_H}
		Consider problem \eqref{eq:constrained_dir_prob2b}. If $H_k$ is a symmetric positive definite matrix, then problem \eqref{eq:constrained_dir_prob2b} admits an optimal solution.
	\end{proposition}
	\begin{proof}
		The quadratic objective function is strictly convex and hence the problem admits a unique solution.
	\end{proof}

	Now, assume $H_k$ is obtained, starting from an $n\times n$ matrix $B_k$, according to \eqref{eq:Hk_explicit}. We can state the following condition. 
	
	\begin{proposition}
		\label{prop:existence_B_new}
		Consider problem \eqref{eq:constrained_dir_prob}.
		If $g_k$ and $s_k$ are non-zero vectors and $B_k$ is a symmetric positive definite matrix, then problem \eqref{eq:constrained_dir_prob} admits an optimal solution.
	\end{proposition}
	\begin{proof}
			Let $H_k$ be the matrix defined in \eqref{eq:Hk_explicit}. Let us consider problem \eqref{eq:constrained_dir_prob2b} 
			as equivalent reformulation of problem \eqref{eq:constrained_dir_prob}, i.e.,
			\begin{equation*}
				\begin{aligned}
					\min_{u} q(u)=\;&  \frac{1}{2}u^TP_k^T B_kP_ku + g_k^T  P_ku. 
				\end{aligned}
			\end{equation*}
			Let us first suppose that $g_k$ and $s_k$ are linearly independent. By definition, $H_k=P_k^TB_kP_k$. Since $B_k$ is positive definite, and $P_k = [-g_k\quad s_k]$ is full rank, $H_k$ is positive definite; then, by \cref{prop:existence_H}, problem \eqref{eq:constrained_dir_prob2b} admits solution.
			
			\smallskip
			On the other hand, let us assume that $g_k$ and $s_k$ are linearly dependent.
			In this case, we can write $s_k=\sigma g_k$ and it results that the Hessian matrix of the quadratic function $q(u)$
			is symmetric positive semi-definite. However,
			we can show that system 
			\begin{equation}\label{syst}
				P_k^T B_kP_ku=-P_k^Tg_k
			\end{equation}
			admits solution, i.e.,
			that there exists at least a point $\bar u$ such that $\nabla q(\bar u)=0$, and hence that $\bar u$ is a global minimizer of the convex function $q(u)$.
			Indeed, we have
			\begin{equation*}
				P_k^TB_kP_k = \begin{bmatrix}
					g_k^TB_kg_k& -\sigma g_k^TB_kg_k\\ 
					-\sigma g_k^TB_kg_k& \sigma^2g_k^TB_kg_k
				\end{bmatrix},
				\qquad
				-P_k^Tg_k=
				\begin{bmatrix}
					\|g_k\|^2 \\ 
					-\sigma \|g_k\|^2
				\end{bmatrix}.
			\end{equation*}
			By the positive definiteness of $B_k$, we have that $g_k^\top B_k g_k > 0$ and we can then write that
			$$
			1=\text{rank}\left(P_k^TB_kP_k\right)=\text{rank}\left(\left[P_k^TB_kP_k \quad -P_k^Tg_k\right]\right).
			$$
			This implies that system \eqref{syst} admits a solution, which is also a solution of problem \eqref{eq:constrained_dir_prob2b} and thus \eqref{eq:constrained_dir_prob}.
	\end{proof}
	
\section{Properties of the search directions}
\label{sec:grad_rel_dir}
In this section we consider again subproblem
\eqref{eq:constrained_dir_prob} by means of the equivalent reformulations
 \eqref{eq:constrained_dir_prob2}-\eqref{eq:constrained_dir_prob2b} and
we assume that it admits solution $u_k=[\alpha_k\;\;\beta_k]^T$, i.e., that at least one assumption
of the preceding section either on $B_k$ or on $H_k$ holds.
We are interested in studying further conditions on the above matrices to ensure that the obtained search directions
$$
d_k=-\alpha_kg_k+\beta_ks_k
$$
are gradient-related according to the following well-known definition \cite{cartis2022evaluation,cartis2015worst}.
\begin{definition}
	\label{def:gr}
	A sequence of search directions $\{d_k\}$ is \textit{gradient-related} to the sequence of solutions $\{x_k\}$ if there exist $c_1>0$ and $c_2>0$ such that, for all $k$, we have
	\begin{equation}
		\label{eq:gr_toint}
		g_k^T d_k\le-c_1\|g_k\|^2, \qquad \|d_k\|\le c_2\|g_k\|.
	\end{equation}
\end{definition}
\par\medskip\noindent
As already said, the above property is a requirement sufficient to guarantee that the sequence generated according to the following scheme
$$
x_{k+1}=x_k+\eta_k d_k,
$$
where $\eta_k$ is the stepsize computed by an Armjio-type line search, is globally convergent to stationary points. Moreover, under this assumption the algorithm can be proved to have a worst case iteration and function evaluations complexity bound of
${\cal O}(\epsilon^{-2})$ to reach a solution $\bar{x}$ with $\|\nabla f(\bar{x})\|\le \epsilon$. Note that this bound is actually tight for first-order methods under standard first-order smoothness assumptions. 
 
\subsection{Conditions on ${B_k}$}\label{sec:nn}

First we state the following result concerning the $n\times n$ matrix $B_k$. Note that, from here onward, we will denote by $\lambda_{\min}(A)$ and 
$\lambda_{\max}(A)$ the minimum and maximum eigenvalues respectively of a symmetric matrix $A$.  

\begin{proposition}\label{teorembk}
  Let $d_k$ be defined by
 $d_k=-\alpha_k g_k +\beta_k s_k$,
where $[\alpha_k\;\;\beta_k]^T$ is solution of \eqref{eq:constrained_dir_prob2}.

  \label{prop:gr_B}
		Let $\{B_k\}\subseteq\mathbb{R}^{n\times n}$ be the sequence of symmetric matrices 
defining problems (\ref{eq:constrained_dir_prob2})
and assume that
		there exist scalars $0<\eta_1\le\eta_2$ such that
		\begin{equation}\label{ipotesi}
			\eta_1\le \lambda_{\min}(B_k)\le \lambda_{\max}(B_k)\le \eta_2
		\end{equation}
holds for all $k$. Then the direction $d_k$ satisfies the following conditions:
	\begin{eqnarray*}
		&&g_k^Td_k\le -\frac{\eta_1}{\eta_2^2}\|g_k\|^2\\
  &&\frac{1}{\eta_2} \|g_k\|\le  \|d_k\|\le\frac{2}{\eta_1}\|g_k\|.
	\end{eqnarray*}

	\end{proposition}
	\begin{proof}
		Let us recall  problem (\ref{eq:constrained_dir_prob2}) 
		\begin{equation*}
			\begin{aligned}
				\min_{u}\;&   \frac{1}{2}u^TP_k^T B_kP_ku +g_k^T  P_ku,
			\end{aligned}
		\end{equation*}
		where 
		$
		d=P_ku
		$, $
		P_k=[-g_k\quad s_k]$ and $u=[\alpha\; \beta]^T$.
				
		From \cref{prop:existence_B_new} we have that there exists at least a solution $u_k$ of problem \eqref{eq:constrained_dir_prob}, satisfying the following linear system
		\begin{equation}\label{systemb}
			P_k^TB_kP_ku_k=-P_k^Tg_k.
		\end{equation}
		Multiplying both the members in \eqref{systemb} by $u_k^T$, we obtain $u_k^TP_k^TB_kP_ku_k=-u_k^TP_k^Tg_k$,
		i.e.,  
		\begin{equation}\label{relation1}
			d_k^TB_kd_k=-g_k^Td_k,
		\end{equation}
		being $d_k=P_ku_k$.
		Now, consider the first row of system \eqref{systemb}, i.e.,
		\begin{equation*}\label{system2-A}
			-g_k^TB_kd_k=\|g_k\|^2.
		\end{equation*}
		Using \eqref{relation1} and recalling assumption \eqref{ipotesi}, we can write
		$$
		\|g_k\|^2=|g_k^TB_kd_k|\le \|g_k\|\|B_k\|\|d_k\|\le \lambda_{\max}(B_k)\|g_k\|\|d_k\|\le \eta_2\|g_k\|\|d_k\|,
		$$
		and hence we obtain
		\begin{equation}\label{relation2}
			\|d_k\|\ge {\frac{1}{\eta_2}}\|g_k\|.
		\end{equation}
		Recalling \eqref{relation1}, using \eqref{ipotesi} and  (\ref{relation2}), it follows
		\begin{equation}\label{final2}
			-g_k^Td_k\ge \lambda_{\min}(B_k)\|d_k\|^2\ge \eta_1\|d_k\|^2\ge 
			{\frac{\eta_1}{\eta_2^2}}
			\|g_k\|^2.
		\end{equation}
		Considering again \eqref{relation1} we can also write
		$$
		\|g_k\|\|d_k\|\ge -g_k^Td_k=d_k^TB_kd_k\ge \lambda_{\min}(B_k)\|d_k\|^2\ge \eta_1\|d_k\|^2,
		$$
		so that we have
		\begin{equation}\label{final1}
			\|d_k\|\le {\frac{1}{\eta_1}}\|g_k\|.
		\end{equation}
		Then, \eqref{relation2}, \eqref{final2} and \eqref{final1} prove the thesis of the proposition.
	\end{proof}

\subsection{Conditions on $H_k$}
\label{sec:22}

In the previous section, problem \eqref{eq:constrained_dir_prob} was rewritten as 
\begin{equation*}
	\begin{aligned}
		\min_{\alpha, \beta}\;&   \frac{1}{2}\ \begin{bmatrix}
			\alpha\\ 
			\beta
		\end{bmatrix}^TP_k^T B_kP_k\begin{bmatrix}
		\alpha\\ 
		\beta
	\end{bmatrix} +g_k^T  P_k\begin{bmatrix}
	\alpha\\ 
	\beta
\end{bmatrix},
	\end{aligned}
\end{equation*}
having set
 $
P_k=\begin{bmatrix}-g_k & s_k\end{bmatrix}$. 
\par\noindent
Then, \cref{prop:gr_B} shows that condition \cref{ipotesi} on the eigenvalues of the matrices $B_k$ guarantees that the solution $[\alpha_k\;\;\beta_k]^T$ of \eqref{eq:constrained_dir_prob2b} produces directions
$d_k=-\alpha_k g_k +\beta_k s_k$ that are gradient-related. 
However, checking or ensuring that a sequence of $n\times n$ matrices $\{B_k\}$ have uniformly bounded eigenvalues can be extremely challenging for large-dimensional optimization problems.
\par
Problem \eqref{eq:constrained_dir_prob2b} can be rewritten in the following form:
\begin{equation}\label{problem-Hk-new}
	\begin{aligned}
		\min_{\alpha, \beta}\;&   \frac{1}{2}\ \begin{bmatrix}
			\alpha\\ 
			\beta
		\end{bmatrix}^TH_k\begin{bmatrix}
			\alpha\\ 
			\beta
		\end{bmatrix} +\begin{bmatrix}
		-\|g_k\|^2\\ 
		g_k^Ts_k
	\end{bmatrix}^T \begin{bmatrix}
			\alpha\\ 
			\beta
		\end{bmatrix}.
	\end{aligned}
\end{equation}

In this section we analyze conditions on the $2\times 2$ matrix $H_k$ sufficient to guarantee suitable properties to the direction $d_k$.

This analysis is fundamental for two reasons. The first one
is that, as said before, these ``direct" conditions on the low-dimensional matrix can be checked and ensured (possibly by suitable modifications) regardless of the $n\times n$ matrix $B_k$.
The second reason is that they can be employed in connection
with any $2\times 2$ matrix defining a quadratic model --- see \eqref{QuadModel2} --- to be minimized for determining the values $\alpha_k$ and $\beta_k$ that characterize the update rule of gradient methods with momentum.
\par
Following the latter point,
 we assume that $H_k$ is any symmetric positive definite matrix.
Then, (\ref{eq:Hk_explicit}) is a particular case of $H_k$ provided that $B_k$ is positive definite and that $g_k$ and $s_k$ are linearly independent.

\par
A possible approach to define conditions to impose on the sequence of  matrices 
 $\{H_k\}$
might, in principle, draw inspiration from \cref{prop:gr_B},  concerning the sequence of matrices $\{B_k\}$. 
First, we therefore state the following theorem involving the sequence of  matrices 
 $\{H_k\}$.
\begin{proposition}
	\label{prop:gr_Hk}
	Let $d_k$ be defined by
	$
		d_k=-\alpha_k g_k +\beta_k s_k,
	$
	where
	$[\alpha_k\;\;\beta_k]^T$ is the solution of (\ref{problem-Hk-new}).
	\par
	Let $\{H_k\}\subseteq\mathbb{R}^{2\times 2}$ be the sequence of symmetric  matrices defining problems (\ref{problem-Hk-new}) and assume that
	there exist scalars $0< \hat{c}_1 \le \hat{c}_2$ such that 
	\begin{equation}\label{assT1-biss}
		\hat{c}_1 \le \lambda_{\min}(H_k)\le \lambda_{\max}( H_k)\le  \hat{c}_2
	\end{equation}
	holds for all $k$.
	
	Then the direction $d_k$ satisfies the following conditions:
	\begin{eqnarray}\label{thesis-tilde H1}
		&&g_k^Td_k\le -\frac{1}{\hat{c}_2}\|g_k\|^4\\
		\label{thesis-tilde H2}
		&&\frac{1}{\hat{c}_2} \|g_k\|^3\le  \|d_k\|\le \frac{1}{\hat c_1} (\| 
		g_k\|+\|s_k\|
		) (
		\| g_k\|^2+
		|g_k^T s_k|).
	\end{eqnarray}
\end{proposition}	
\begin{proof} 
	We know that $\alpha_k$ and $\beta_k$ are such that:
	\begin{equation}\label{sistem-Hk}
		\begin{bmatrix}
			\alpha_k\\
			\beta_k
		\end{bmatrix}\ =\ { H_k}^{-1} \begin{bmatrix}
			\| g_k\|^2\\
			-g_k^T s_k
		\end{bmatrix}.
	\end{equation}
	Multiplying both sides of the above equation by $\begin{bmatrix}
		\| g_k\|^2& -g_k^T s_k\end{bmatrix}
	$  we obtain:
	\begin{equation}\label{start2x2-biss}
		\alpha_k\|g_k\|^2 - \beta_kg_k^T s_k = 
		\begin{bmatrix}\|g_k\|^2 &  -g_k^T s_k\end{bmatrix} H_k^{-1} \begin{bmatrix}
			\| g_k\|^2\\
			-g_k^T s_k
		\end{bmatrix}. 
	\end{equation}
	We note  that $-g_k^T d_k  = \alpha_k\|g_k\|^2 - \beta_kg_k^T s_k$; thus, by using \cref{assT1-biss} and \eqref{start2x2-biss} we have:
	$$
	-g_k^T d_k  
	\geq \lambda_{min}( H^{-1}_k) 
	\biggl(\|g_k\|^4+ (g_k^T s_k)^2\biggr),
	$$
	which implies:
	$$
	g_k^T d_k  
	\leq -\frac{1}{\hat{c}_2}
	\|g_k\|^4,
	$$
	and thus proves  \eqref{thesis-tilde H1}.
	
	\smallskip\noindent
	The previous inequality and the Schwarz inequality imply
	$$\|d_k\|\ge \frac{1}{\hat{c}_2} \|g_k\|^3$$
	which proves the first inequality of (\ref{thesis-tilde H2}).
	\par
	By using again  \eqref{sistem-Hk}, we can get an upper bound on the norm of $d_k$:
	\begin{eqnarray*}
		&& \|d_k\|=\biggl\|\begin{bmatrix}
			-g_k& s_k
		\end{bmatrix} \begin{bmatrix}
			\alpha_k\\
			\beta_k
		\end{bmatrix}\ \biggr\|
		=\biggl\|\ \begin{bmatrix}
			-g_k& s_k
		\end{bmatrix} { H_k}^{-1} \begin{bmatrix}
			\| g_k\|^2\\
			-g_k^T s_k
		\end{bmatrix}\biggr\|\\
		&&\qquad\ \le \biggl\| { H_k}^{-1} \biggr\| \ 
		\biggl\| \begin{bmatrix}
			-g_k& s_k
		\end{bmatrix} \biggr\| \ 
		\biggl\| \begin{bmatrix}
			\| g_k\|^2\\
			-g_k^T s_k
		\end{bmatrix} \biggr\|\\
		&&\qquad\ \le \frac{1}{\hat c_1} (\| 
		g_k\|+\|s_k\|
		) (
		\| g_k\|^2+
		|g_k^T s_k|
		),
	\end{eqnarray*}
	which proves the second  relation of (\ref{thesis-tilde H2}) and completes the proof.
\end{proof}

According to the above result, suitable descent properties of the obtained search directions hold and
would in fact be sufficient, coupled with the employment of a standard Armjio-type line search,
to define a globally convergent gradient method with momentum. 

However, as we detail below, the obtained sequence of directions $\{d_k\}$ is not gradient-related according to \cref{def:gr} and, hence, an ${\cal O}(\epsilon^{-2})$ complexity bound cannot be ensured.
\begin{remark}
We observe that when the matrix $H_k$ is given by (\ref{eq:Hk_explicit}), it tends to become the null matrix as $x_k$ approaches a stationary point. Hence,
uniform boundedness conditions on the eigenvalues of $H_k$ are in contrast with the matrices deriving from \eqref{eq:constrained_dir_prob}.

\noindent Even more in general,
under the assumptions of \cref{prop:gr_Hk}, it is actually not possible to
ensure that the obtained sequence of directions $\{d_k\}$ is gradient-related. Indeed, consider the cases where
 $g_k^T s_k=0$. The inequality \eqref{thesis-tilde H2} implies:
$$
|g_k^T d_k|\le   \|g_k\| \|d_k\|\le \frac{1}{\hat c_1} (\| 
		g_k\|+\|s_k\|
	) \| g_k\|^3,
$$
and hence, by assuming that $\{s_k\}$ is bounded, for sufficiently small values of $\|g_k\|$, the direction $d_k$ does not satisfy the first requirement of (\ref{eq:gr_toint}).
\end{remark}
\par\medskip\noindent
By the next proposition, we finally state suitable conditions on the 
sequence  of matrices $\{H_k\}$, that also take into account
the sequences of vectors involved, i.e., $\{g_k\}$ and $\{s_k\}$. These conditions {represent the main theoretical contribution of the paper, as they} are sufficient to ensure that the obtained sequence of search directions $\{d_k\}$ is gradient-related. 
	\begin{proposition}
 \label{prop:gr_tilde Hk}
Let $\{H_k\}\subseteq\mathbb{R}^{2\times 2}$ be the sequence of symmetric  matrices defining problems (\ref{problem-Hk-new}) and assume that
	there exist scalars $0< c_1 \le c_2$ such that
	\begin{equation}\label{assT1bis}
		 c_1 \le \lambda_{\min}(D_k^{-1} H_kD_k^{-1})\le \lambda_{\max}(D_k^{-1}H_kD_k^{-1})\le  c_2
	\end{equation}
 holds for all $k$, where
 $$D_k=\begin{bmatrix}
 	\|g_k\|& 0\\ 
 	0& \|s_k\|
 \end{bmatrix}.$$
  Let $d_k$ be defined by
 $d_k=-\alpha_k g_k +\beta_k s_k$,
where $[\alpha_k\;\;\beta_k]^T$ is solution of problem (\ref{problem-Hk-new}).
 Then, the direction $d_k$ satisfies the following conditions:
	\begin{eqnarray}\label{thesis-tilde H1new}
		&&g_k^Td_k\le -\frac{1}{c_2}\|g_k\|^2\\
  \label{thesis-tilde H2new}
  &&\frac{1}{c_2} \|g_k\|\le  \|d_k\|\le\frac{2}{c_1}\|g_k\|.
	\end{eqnarray}
\end{proposition}	
\begin{proof}
For simplicity, let us introduce the following matrix:
\begin{gather*}
    \label{matr1}
    \tilde H_k=D_k^{-1}H_kD_k^{-1}=\begin{bmatrix}
		\frac{ (H_{11})_k}{\|g_k\|^2}& \frac{(H_{12})_k}{\|s_k\|\,\|g_k\|}\\ 
		\frac{ (H_{12})_k}{\|s_k\|\,\|g_k\|}& \frac{ (H_{22})_k}{\|s_k\|^2}
	\end{bmatrix}
\end{gather*}
Then, problem \eqref{problem-Hk-new} can be rewritten as:
\begin{equation*}\label{problem-tilde-Hk-new}
	\begin{aligned}
		\min_{\alpha, \beta}\;&   \frac{1}{2}\ \begin{bmatrix}
			\alpha\\ 
			\beta
		\end{bmatrix}^T D_k \tilde H_k D_k\begin{bmatrix}
			\alpha\\ 
			\beta
		\end{bmatrix} +\begin{bmatrix}
			-\|g_k\|^2\\ 
			g_k^Ts_k
		\end{bmatrix}^T \begin{bmatrix}
			\alpha\\ 
			\beta
		\end{bmatrix},
	\end{aligned}
\end{equation*}
and its optimal solutions
 $\alpha_k$ and $\beta_k$ are such that
		\begin{equation}\label{start2x2}
D_k \tilde H_k D_k
\begin{bmatrix}
	\alpha_k\\
	\beta_k
\end{bmatrix}\ =\ \begin{bmatrix}
	\| g_k\|^2\\
	-g_k^T s_k
\end{bmatrix},
	\end{equation}
from which we have:
		\begin{equation}\nonumber
 D_k
	\begin{bmatrix}
		\alpha_k\\
		\beta_k
	\end{bmatrix}\ =\ {\tilde H_k}^{-1} D_k^{-1}\begin{bmatrix}
		\| g_k\|^2\\
		-g_k^T s_k
	\end{bmatrix},
\end{equation}
namely:
$$
\begin{bmatrix}
	\alpha_k\|g_k\|\\
	\beta_k\|s_k\|
\end{bmatrix}\ =\ {\tilde H_k}^{-1}\begin{bmatrix}
	\| g_k\|\\
	-\frac{g_k^T s_k}{\|s_k\|}
\end{bmatrix}.
$$
Multiplying both sides of the above equation by $\begin{bmatrix}
	\| g_k\|& -\frac{g_k^T s_k}{\|s_k\|}\end{bmatrix}
	$  we obtain:
\begin{equation}\label{start2x2-bis}
 \alpha_k\|g_k\|^2 - \beta_kg_k^T s_k = 
\begin{bmatrix}\|g_k\| &  -\frac{g_k^T s_k}{\|s_k\|}\end{bmatrix} \tilde H_k^{-1} \begin{bmatrix}
	\| g_k\|\\
	-\frac{g_k^T s_k}{\|s_k\|}
\end{bmatrix}. 
\end{equation}
Recalling \cref{assT1bis} and that $-g_k^T d_k  = \alpha_k\|g_k\|^2 - \beta_kg_k^T s_k$, we have from (\ref{start2x2-bis}):
$$
-g_k^T d_k  
\geq \lambda_{min}(\tilde H^{-1}) 
\biggl(\|g_k\|^2+ \left(\frac{g_k^T s_k}{\|s_k\|}\right)^2\biggr),
$$
 and thus
\begin{eqnarray} \label{newcond1tris}
&&g_k^T d_k  
\leq -\frac{1}{c_2}
\|g_k\|^2.
\end{eqnarray}
Then, Schwarz inequality implies:
\begin{equation}\label{newcond1}
	 \| g_k\|\le  c_2 \, \|d_k\|.
\end{equation}
Now,  multiplying  both terms of equality \eqref{start2x2}
by $\begin{bmatrix}
	\alpha_k&
	\beta_k
\end{bmatrix}$ we obtain:
		\begin{equation*}\label{start2x2-tris}
\begin{bmatrix}
	\alpha_k \| g_k\|\\
	\beta_k \| s_k\|
\end{bmatrix}^T \tilde H_k 
	\begin{bmatrix}
		\alpha_k \| g_k\|\\
		\beta_k \| s_k\|
	\end{bmatrix}\ = -g_k^T(-\alpha g_k+\beta_k s_k)=-g_k^T d_k,
\end{equation*}
from which we get:
		\begin{equation*}\label{newcond2}
	\begin{aligned}
		-g_k^T d_k & \ge \lambda_{min}\left(\tilde{H}_k\right)(\alpha_k^2\|g_k\|^2 + \beta_k^2\|s_k\|^2) \\&\geq c_1(\alpha_k^2\|g_k\|^2 + \beta_k^2\|s_k\|^2)
  \\&\ge \frac{c_1}{2}\|d_k\|^2.
	\end{aligned}
\end{equation*}
Then, by using Schwarz inequality we get:
$$
\|g_k\|\|d_k\|\geq|g_k^T d_k| \geq \frac{c_1}{2}\|d_k\|^2
$$
and hence
\begin{equation} \label{newcond1bis}
\|d_k\| \leq \frac{2}{c_1}\|g_k\|.
\end{equation}
Now (\ref{newcond1tris}), (\ref{newcond1}) and (\ref{newcond1bis}) imply  (\ref{thesis-tilde H1new}) and (\ref{thesis-tilde H2new})  and, hence, 
the proof is complete.
\end{proof}
\begin{remark}\label{remarknew}
It is important to note that  assumption \eqref{assT1bis} is not difficult to satisfy.
In fact, given any sequence of symmetric matrices 
$\{\hat H_k\}\subseteq\mathbb{R}^{2\times 2}$ for which 
	there exist scalars $0< \tilde c_1 \le \tilde{c}_2$ such that, for all $k$,
	\begin{equation}\label{assT1tris}
		 \tilde{c}_1 \le \lambda_{\min}(\hat H_k)\le \lambda_{\max}(\hat H_k)\le \tilde c_2,
	\end{equation}
a sequence of matrices $\{ H_k\}$ satisfying (\ref{assT1bis}) is obtained by choosing:
$$
H_k = \left[\begin{array}{cc}
		\|g_k\|(\hat H_{11})_k \|g_k\|& \|g_k\| (\hat H_{12})_k \|s_k\| \\
		\|s_k\|(\hat H_{21})_k \|g_k\|& \|s_k\|(\hat H_{22})_k \|s_k\|
	\end{array}\right].
$$
{We will show later in this work how the condition can practically be enforced in computational scenarios.}
\end{remark}
\section{Algorithmic model}
\label{sec:alg}
In this section we present an algorithmic framework for gradient methods with momentum which exploits the  theoretical analysis carried out in the previous sections. The idea is to define a first-order algorithm with both strong theoretical guarantees in the nonconvex setting and the possibility of computationally exploiting eventual second-order information on the minimization problem. Before formally presenting the algorithm, we summarize the key steps:
\begin{itemize}
\item[(a)] define a quadratic subproblem
\begin{equation}\label{problem-Hk-new_alg}
	\begin{aligned}
		\min_{\alpha, \beta}\;&   \frac{1}{2}\ \begin{bmatrix}
			\alpha\\ 
			\beta
		\end{bmatrix}^TH_k\begin{bmatrix}
			\alpha\\ 
			\beta
		\end{bmatrix} +\begin{bmatrix}
		-\|g_k\|^2\\ 
		g_k^Ts_k
	\end{bmatrix}^T \begin{bmatrix}
			\alpha\\ 
			\beta
		\end{bmatrix}.
	\end{aligned}
\end{equation}
where $H_k$ is a $2\times 2$ symmetric matrix;
\item[(b)] once computed a solution $[\alpha_k \ \beta_k]^T$ of (\ref{problem-Hk-new_alg}), provided one exists, a check on the obtained search direction 
$$
d_k=-\alpha_kg_k+\beta_ks_k
$$
is performed in order to ensure the gradient-related property of the sequence $\{d_k\}$;
\item[(c)] if the test is satisfied, then a standard Armjio-type line search is performed along $d_k$; otherwise, a suitable modification of $H_k$ based on (\ref{assT1bis}) is introduced and, again, steps (a) (with the modified $H_k$) and (b) (without the check on $d_k$) are performed,  as well as the Armjio-type line search  along the obtained $d_k$.
\end{itemize}
\par
\medskip\noindent
The proposed Algorithmic Model is described in Algorithm \ref{alg:qps}. Notice that the initial tentative stepsize $\eta=1$ is optimal according to the quadratic model used to define the search direction $d_k$; thus, the unit step will often be a good step even for the true objective and satisfy the Armijo sufficient decrease condition. From a computational perspective, this allows to possibly save several backtracking steps and, consequently, function evaluations.
\begin{algorithm}[htbp]
	\caption{\texttt{Gradient Method with Momentum (GMM)}}
	\label{alg:qps}
	\algnewcommand{\LineComment}[1]{\State \texttt{\( \slash\ast \) #1 \( \ast\slash \)}}
	\begin{algorithmic}[1]
		\State Input: $x_0\in\mathbb{R}^n$, $\gamma\in(0,1)$, $\delta\in(0,1)$, $c_1>0$, $c_2>0$.
		\State Set $k\leftarrow 0$
		\While{$\|\nabla f(x_k)\|> 0$}
			\LineComment{Compute the search direction}
			\State Set \texttt{\textit{gr\_dir\_found}} $\leftarrow$ \texttt{False}
			\State Define a $2\times 2$ symmetric matrix $H_k$ \label{step:first_H}
			\If{problem \eqref{problem-Hk-new_alg} admits solution $[\alpha_k\;\;\beta_k]^T$}
			\State Set $
				d_k\leftarrow -\alpha_kg_k+\beta_k s_k
				$
				\If{$g_k^Td_k\le  -c_1\|g_k\|^2$ \textbf{and} $\|d_k\|\le  c_2 \|g_k\|$ \label{step:test}}
				\State Set \texttt{\textit{gr\_dir\_found}} $\leftarrow$ \texttt{True}
				\EndIf
				\EndIf	
			
			\If{\texttt{gr\_dir\_found} $=$ False}
			\State Define a new $2\times 2$ symmetric  matrix $H_k$ satisfying condition  \eqref{assT1bis} \label{step:new_H}
			\State Compute $\alpha_k$ and $\beta_k$ by solving \eqref{problem-Hk-new_alg}
				
				\State Set  $d_k\leftarrow -\alpha_kg_k+\beta_k s_k$
				\label{step:new_dir}
			\EndIf 
			
			%
			%
			\LineComment{Perform Armijo line search along $d_k$}
			
			%
			\State Set $\eta \leftarrow 1$
			\While{$f(x_k+\eta d_k) > f(x_k) + \gamma \eta d_k^\top\nabla f(x_k)$}
					\State Set $\eta \leftarrow \delta\eta$
			\EndWhile
			\State Set $\eta_k \leftarrow \eta$,
			%
			%
			$x_{k+1} \leftarrow x_k + \eta_kd_k$
			\State Set $k\leftarrow k+1$
		\EndWhile
		%
	\end{algorithmic}
\end{algorithm}

%

The theoretical properties of the proposed framework for gradient methods with momentum are stated in the following proposition.

\begin{proposition}

	Assume that $\mathcal{L}_0 = \{x\in\Re^n\mid f(x)\leq f(x_0)\}$ is a compact set. 
 Then, Algorithm \ref{alg:qps} either stops in a finite number of iterations $\nu$ producing a point $x_\nu$ which is stationary for $f$, i.e. $\nabla f(x_\nu) = 0$, or it produces an infinite sequence $\{x_k\}$ that admits limit points, each one being a stationary point for $f$. Furthermore, if the gradient $\nabla f$ is Lipschitz continuous on $\Re^n$, we have that Algorithm \ref{alg:qps} requires at most $\mathcal{O}(\eps^{-2})$ iterations, function and gradient evaluations to attain
	$$
	\|\nabla f(x_k)\|\le \epsilon_k.
	$$
\end{proposition}
\begin{proof}
From the steps the algorithm and Proposition \ref{prop:gr_tilde Hk} we have that
 $\{d_k\}$ is a sequence of gradient-related directions. Since the Armijo line search is employed within Algorithm \ref{alg:qps}, the results follow from \cite[Proposition 1.2.1]{bertsekas1999nonlinear} and by \cite{cartis2015worst}.
\end{proof}

We now focus on the issue of modifying a given
$2\times 2$ matrix $H_k$ in order to satisfy condition
\eqref{assT1bis} as required at step \ref{step:new_H} of Algorithm \ref{alg:qps}.
Suppose that a symmetric $H_k$ matrix has been defined at step \ref{step:first_H}, but either problem \eqref{problem-Hk-new_alg} does not admit solution or the test at step \ref{step:test} is not satisfied, i.e., the value of \texttt{\textit{gr\_dir\_found}} remains False. Then, step \ref{step:new_H}  must be performed and a new matrix $H_k$ must be constructed modifying as least as possible the given matrix
defined at step \ref{step:first_H}.
 
Let us denote by $H_k^0$ the matrix defined at step \ref{step:first_H} and by $H_k$ the new matrix defined at step \ref{step:new_H}.
We can proceed as follows:
\begin{itemize}
\item[-] A matrix $\hat H_k$ can be obtained by a modified Cholesky factorization \cite{bertsekas1999nonlinear} applied to $H_k^0$ 
\item[-] Letting $$
D_k=\begin{bmatrix}
	   \|g_k\|& 0\\ 
	   0& \|s_k\|
    \end{bmatrix}
 $$
we can set $H_k=D_k\hat H_k D_k$,
i.e.,
$$
H_k = \left[\begin{array}{cc}
		\|g_k\|(\hat H_{11})_k \|g_k\|& \|g_k\| (\hat H_{12})_k \|s_k\| \\
		\|s_k\|(\hat H_{21})_k \|g_k\|& \|s_k\|(\hat H_{22})_k \|s_k\|
	\end{array}\right].
$$
\end{itemize}
The boundedness of $\{H_k^0\}$ and the properties of the 
modified Cholesky factorization imply that
\eqref{assT1tris} of 
Remark \ref{remarknew} holds,  so that, according to the same remark, the matrix $H_k$ is such that
condition
\eqref{assT1bis} is satisfied.
\par
\smallskip
We conclude this section by stating a theoretical result showing that,
under suitable assumptions on the objective function and a choice of the matrix $H_k$ actually related to the Hessian $\nabla^2f(x_k)$,
steps \ref{step:new_H}-\ref{step:new_dir} are never executed for $k$ sufficiently large, that is,
the test at step \ref{step:test} is always satisfied.

\begin{proposition}
Let $f:\Re^n\to \Re$ be a twice continuously differentiable function.
Suppose that
$$
H_k=
\begin{bmatrix}
	g_k^T\nabla^2f(x_k)g_k& -g_k^T\nabla^2f(x_k)s_k\\ 
	-g_k^T\nabla^2f(x_k)s_k& s_k^T\nabla^2f(x_k)s_k
\end{bmatrix}
$$
is the matrix defined at step \ref{step:first_H},  and let $\{x_k\}$ be the sequence generated by Algorithm \ref{alg:qps}.
Assume that $\{x_k\}$ converges to $x^*$, where $\nabla f(x^*)=0$ and the Hessian matrix $\nabla^2f(x^*)$ is positive definite.
Furthermore, assume that the constants of the test at step \ref{step:test} are chosen in a such a way that:
 \begin{equation*}\label{NWtipe} c_1\le \theta^3\frac{\lambda_{\min}(\nabla^2 f(x^*))}{\lambda_{\max}(\nabla^2 f(x^*))^2},\qquad\qquad c_2\ge \frac{2}{\theta\lambda_{\min}(\nabla^2 f(x^*))}. \end{equation*}
 where $\theta\in (0,1)$.
 Then, for $k$ sufficiently large the test at step \ref{step:test} is satisfied.  
\begin{proof}
By continuity there exists a neighborhood ${\cal B}(x^*)$ of $x^*$ such that if $x_k\in{\cal B}(x^*) $ we have:
$$\theta\lambda_{\min}(\nabla^2 f(x^*))\le \lambda_{\min}(\nabla^2 f(x_k)),\qquad \lambda_{\max}(\nabla^2 f(x_k))\le \frac{1}{\theta}\lambda_{\max}(\nabla^2 f(x^*)).$$
Then the  thesis follows from Proposition \ref{teorembk} by setting $$B_k=\nabla^2 f(x_k), \quad\eta_1=\theta\lambda_{\min}(\nabla^2 f(x_k)), \quad \eta_2=\frac{1}{\theta}\lambda_{\max}(\nabla^2 f(x_k))$$
and the proof is complete.
\end{proof}
\end{proposition}


\section{Computation of ${H_k}$}\label{sec:comp_H}

The core of the general framework lies in
how a sequence of $2\times 2$ matrices $\{H_k\}$ can be determined to ensure an efficient computational behavior of the algorithm, as well as exploiting \cref{prop:gr_B} or \cref{prop:gr_tilde Hk} to guarantee sound theoretical properties.\par\noindent
Regarding the first issue, the following two strategies can be adopted: 
\begin{itemize}
\item[i)]  define a suitable $n\times n$ matrix $B_k$  and compute the matrix $H_k$ by using (\ref{eq:Hk_explicit});
\item[ii)]  define a suitable $2\times 2$ matrix $H_k$ independent on any $n\times n$ matrix.
\end{itemize}
\par\smallskip
Concerning strategy i), in large-scale optimization problems the use and storage of the $n\times n$ matrix $B_k$ can be computationally too expensive, if not downright prohibitive. Therefore, to take into account this issue, we propose two approaches described in \cref{sec:Hd} and \cref{sec:diag}. A technique related to strategy ii) is presented in \cref{sec:interp}.
Summarizing, we propose three techniques to compute the matrix $H_k$, although
other approaches could be exploited within the general framework
we have presented.
\subsection{Approximating Hessian-vector products  by finite differences of gradients}
\label{sec:Hd}
In line with the strategy employed in \cite{zhang2022drsom,tang2024riemannian}, we can draw inspiration by the strategy employed in Truncated Newton methods  \cite{grippo1989truncated,dembo1982inexact}, where the explicit management of the Hessian matrix $\nabla^2 f(x_k)$ is not required, but rather the Hessian-vector product $\nabla^2 f(x_k)d$ is directly handled, with $d\in R^n$.

Consider $H_k$ defined by \eqref{eq:Hk_explicit}, i.e.,
\[
	H_k = \begin{bmatrix}
		g_k^TB_kg_k& -g_k^TB_ks_k\\ 
		-g_k^TB_ks_k& s_k^TB_ks_k
	\end{bmatrix}.
\]
The elements of $H_k$ can be obtained estimating the
two matrix-vector products $B_kg_k$ and $B_ks_k$ by finite difference approximation, namely by the following vectors (where $\xi > 0$ is a suitably small parameter):
\begin{eqnarray*}
\frac{\nabla f(x_k+\xi g_k/\|g_k\|) - g_k}{\xi/\|g_k\|} & \approx & \nabla^2 f(x_k)g_k,\\
\frac{\nabla f(x_k+\xi s_k/\|s_k\|) - g_k}{\xi/\|s_k\|} & \approx& \nabla^2 f(x_k)s_k.
\end{eqnarray*}
In this way  it is possible to consistently construct $H_k$ without the need of handling an $n\times n$ matrix. The price to pay consists in two additional evaluations of the $n$ dimensional gradient of $f$, $\nabla f(x)$.


\subsection{Hessian estimation in subspace by interpolation}\label{sec:interp}

In this subsection,
an alternative way to compute the matrix $H_k$ is proposed that avoids the need of additional $n$ dimensional gradient evaluations. We note that a simplified version of the following approach was found to be particularly well-performing in \cite{zhang2022drsom}.
\par
Let us consider
the two-dimensional function  
$$
\psi_k(\alpha,\beta)=f(x_k-\alpha g_k+\beta s_k).
$$
Assuming that $f$ is twice continuously differentiable,
we have seen --- see \eqref{QuadModel} --- that
the quadratic Taylor polynomial of $\psi_k(\alpha,\beta)$ centered at $(0,0)$ is
\begin{equation*}\label{QuadModel_2}
\psi_k(\alpha,\beta) = f(x_k)+\begin{bmatrix}
		-\|g_k\|^2\\ 
		g_k^Ts_k
	\end{bmatrix}^T \begin{bmatrix}
			\alpha\\ 
			\beta
		\end{bmatrix}
		+\frac{1}{2}
		\begin{bmatrix}
			\alpha\\ 
			\beta
		\end{bmatrix}^T
		\begin{bmatrix}
	g_k^T\nabla^2f(x_k)g_k& -g_k^T\nabla^2f(x_k)s_k\\ 
	-g_k^T\nabla^2f(x_k)s_k& s_k^T\nabla^2f(x_k)s_k
\end{bmatrix}
\begin{bmatrix}
			\alpha\\ 
			\beta
		\end{bmatrix}.
\end{equation*}
Then, the matrix $H_k$ defined by (\ref{eq:Hk_explicit}) can be
viewed as a matrix approximating $\nabla^2\psi_k(0,0)$, and, in a more general setting,
$H_k$ could be any $2\times 2$ matrix (independent of the $n\times n$ matrix $B_k$).
This leads to consider the following approximation of the quadratic Taylor polynomial of $\psi_k(\alpha,\beta)=f(x_k-\alpha g_k+\beta s_k)$:
\begin{equation*}\label{QuadModel2_2}
\phi(\alpha,\beta)=
f(x_k)+\begin{bmatrix}
		-\|g_k\|^2\\ 
		g_k^Ts_k
	\end{bmatrix}^T \begin{bmatrix}
			\alpha\\ 
			\beta
		\end{bmatrix}
		+\frac{1}{2}
		\begin{bmatrix}
			\alpha\\ 
			\beta
		\end{bmatrix}^T
\left[\begin{array}{cc}
		H_{11} & H_{12}  \\
		H_{12} & H_{22} 
	\end{array}\right]		
\begin{bmatrix}
			\alpha\\ 
			\beta
		\end{bmatrix}.
\end{equation*}
The three elements defining a symmetric matrix $H_k$ can be determined by imposing the interpolation conditions on
three points  $(\alpha_1,\beta_1)$, $(\alpha_2,\beta_2)$, and $(\alpha_3,\beta_3)$ different from $(0,0)$:
\begin{gather*}
	\phi(\alpha_1,\beta_1)=\psi_k(\alpha_1,\beta_1)=f(x_k-\alpha_1g_k+\beta_1s_k),\\
	\phi(\alpha_2,\beta_2)=\psi_k(\alpha_2,\beta_2)=f(x_k-\alpha_2g_k+\beta_2s_k),\\
	\phi(\alpha_3,\beta_3)= \psi_k(\alpha_3,\beta_3)=f(x_k-\alpha_3g_k+\beta_3s_k),
\end{gather*}
i.e.,
by solving the linear system
\begin{eqnarray*}
	\label{eq:lin_sys_1bbb}
	&&\lefteqn{\frac{1}{2}\left(\!\!\begin{array}
		{ccc}		\alpha_1^2&2\alpha_1\beta_1&\beta_1^2\\\alpha_2^2&2\alpha_2\beta_2&\beta_2^2\\\alpha_3^2&2\alpha_3\beta_3&\beta_3^2
	\end{array}\!\!\right)\!\!\left(\!\!\begin{array}{c}
		H_{11}\\H_{12}\\H_{22}
	\end{array}\!\!\right) =} \\ &&\qquad\qquad \left(\!\!\begin{array}{c}
		f(x_k-\alpha_{1}g_k + \beta_{1}s_k) - f(x_k) +\alpha_{1}\|g_k\|^2 - \beta_{1}g_k^Ts_k\\
		f(x_k-\alpha_{2}g_k + \beta_{2}s_k) - f(x_k) + \alpha_{2}\|g_k\|^2 - \beta_{2}g_k^Ts_k\\
		f(x_k-\alpha_{3}g_k + \beta_{3}s_k) - f(x_k) + \alpha_{3}\|g_k\|^2 - \beta_{3}g_k^Ts_k
	\end{array}\!\!\right).
\end{eqnarray*}

\begin{remark}
The described strategy requires three function evaluations. However, considering the point $(\alpha_1,\beta_1)=(0,-1)$, we have
$$
f(x_k-\alpha_{1}g_k + \beta_{1}s_k)=
f(x_k-(x_k-x_{k-1})=f(x_{k-1}).
$$
Then, by exploiting the information of the past iteration, i.e., the knowledge of $f(x_{k-1})$, the matrix $H_k$ can be built by only two additional function evaluations.

	Two reasonable candidate points to evaluate the function at might be, for example, $(\alpha_{k-1},\beta_{k-1})$ and $(\alpha_{k-1},0)$. 
The interpolation system to obtain the quadratic matrix $H_k$ becomes
	\begin{eqnarray*}
		&&\lefteqn{\frac{1}{2}\left(\begin{array}
			{ccc}
			0&0&1\\\alpha_{k-1}^2&0&0\\\alpha_{k-1}^2&2\alpha_{k-1}\beta_{k-1}&\beta_{k-1}^2
		\end{array}\right)\left(\begin{array}{c}
			H_{11}\\H_{12}\\H_{22}
		\end{array}\right) =} \\ &&\qquad\qquad\left(\begin{array}{c}
			f(x_{k-1})-f(x_k)+g_k^Ts_k\\
			f(x_k-\alpha_{k-1}g_k) - f(x^k) + \alpha_{k-1}\|g_k\|^2\\
			f(x_k-\alpha_{k-1}g_k + \beta_{k-1}s_k) - f(x_k) + \alpha_{k-1}\|g_k\|^2 - \beta_{k-1}g_k^Ts_k
		\end{array}\right).
	\end{eqnarray*}
\end{remark}
\subsection{ Hessian approximation  by a diagonal matrix}\label{sec:diag}

The last example of computation of the matrix $H_k$ does not require any additional function or gradient  computations. The idea is to  consider a diagonal matrix $B_k$, that is
\begin{equation}\label{Bksimmetrica} B_k=\begin{bmatrix}
		(\mu_k)_1 & 0 &\cdots &0  \\ 
		0 & (\mu_k)_2 &\cdots& 0 \\
         \vdots & \vdots &\ddots&\vdots \\
         0 & \cdots &\cdots&(\mu_k)_n 
	\end{bmatrix}.\end{equation} 
The diagonal elements of the matrix $B_k$ are computed  by drawing inspiration from the  approach of Barzilai-Borwein methods \cite{raydan1997barzilai}. The matrix $B_k$ given by \eqref{Bksimmetrica} is the optimal solution of the following problem:
$$\min_B \|B_k s_k -y_k\|^2$$
where
\begin{gather*}
y_k  =  g_k - g_{k-1},\qquad
s_k =  x_k - x_{k-1}.
\end{gather*}
This implies that, for $i=1,\ldots,n$:
$$(\mu_k)_i =(y_k)_i/(s_k)_i.$$


\noindent
Finally, the elements of the $2\times 2$ matrix $H_k$ are given by:
\begin{gather*}
	(H_{11})_k  = \sum_{i=1}^n (\mu_k)_i (g_k)_i^2, \quad
	(H_{12})_k  = \sum_{i=1}^n (\mu_k)_i (g_k)_i(s_k)_i, \quad
	(H_{22})_k  = \sum_{i=1}^n (\mu_k)_i (s_k)_i^2.
\end{gather*}
\section{Computational experiments}
\label{sec:exp}
In this section, we describe and report the results of thorough computational experiments aimed at assessing the potential of the algorithm proposed in this work.


The code for all the experiments described in this work was written in Python 3.11, exploiting \texttt{numpy} and \texttt{scipy} libraries. The software is available at:
\href{https://github.com/gliuzzi/GMM}{https://github.com/gliuzzi/GMM}.

Regarding Algorithm \ref{alg:qps}, it has been implemented choosing matrix $H_k$ according to the rule discussed in \cref{sec:interp}, which provided better results than those from \cref{sec:Hd} and \cref{sec:diag} in preliminary experiments; \color{black} we employed the safeguarding technique based on the modified Cholesky factorization, described in \cref{sec:alg}, to ensure that Assumption \cref{assT1bis} holds. 
As for the line search, we set $\delta = 0.5, \; \gamma=10^{-5}$. We also allowed a slight degree of non-monotoniticy in the line search, based on the rule proposed in \cite{zhang2004nonmonotone}. The point $x_{-1}$ is set equal to $x_0$, so that the momentum term is null at the first iteration.

As a baseline for comparisons, we took into account classical nonlinear conjugate gradient methods \cite{hager2006survey} and the L-BFGS algorithm \cite{liu1989limited}, {the latter arguably being the most popular solver} for the solution of smooth unconstrained nonlinear optimization problems. For both algorithms, we considered the very popular implementations available through the \texttt{scipy} library. 

{We also took into account the state-of-the-art \texttt{cg\_descent} procedure available at \href{https://people.clas.ufl.edu/hager/software/}{https://people.clas.ufl.edu/hager/software/} and based on \cite{hager2005new,hager2006algorithm,hager2006survey,hager2013limited}. The \texttt{cg\_descent} provides a highly efficient implementation of a conjugate-gradient type method with subspace optimization steps. We used the python wrapper of the original C code of \texttt{cg\_descent} v.\ 6.8, available at \href{https://github.com/alexfikl/pycgdescent}{https://github.com/alexfikl/pycgdescent}, so that we could employ the procedure within the same testing environment as the other considered methods.}

The results of the experiments are shown in the manuscript in the form of performance profiles \cite{dolan2002benchmarking}. {The detailed tables with all the result can be accessed in the code repository \href{https://github.com/gliuzzi/GMM}{https://github.com/gliuzzi/GMM}.}

We considered a benchmark of 77 unconstrained problems from the CUTEst test-suite \cite{cutest} with a number of variables $n\ge 1000$. In Table \ref{tab:problems} these 77 problems are listed along with their dimension; note that, for problems with a user definable number of variables we always choose the default value. \color{black}
\begin{table}[htbp]    
\caption{Collection of 77 unconstrained CUTEst problems ($n$ denotes the number of variables). }
    \label{tab:problems}
\centering
    \begin{tabular}{|l|r|l|r|l|r|}\hline
    Problem & $n$ & Problem & $n$ & Problem & $n$ \\\hline
 ARWHEAD	&	5000	&	DIXMAANM1	&	3000	&	MSQRTBLS	&	1024	\\
BDQRTIC	&	5000	&	DIXMAANN	&	3000	&	NCB20	&	5010	\\
BOX	&	10000	&	DIXMAANO	&	3000	&	NCB20B	&	5000	\\
BOXPOWER	&	20000	&	DIXMAANP	&	3000	&	NONCVXU2	&	5000	\\
BROYDN3DLS	&	5000	&	DIXON3DQ	&	10000	&	NONDIA	&	5000	\\
BROYDN7D	&	5000	&	DQDRTIC	&	5000	&	NONDQUAR	&	5000	\\
BROYDNBDLS	&	5000	&	DQRTIC	&	5000	&	PENALTY1	&	1000	\\
BRYBND	&	5000	&	EDENSCH	&	2000	&	POWELLSG	&	5000	\\
CHAINWOO	&	4000	&	EG2	&	1000	&	POWER	&	10000	\\
COSINE	&	10000	&	EIGENALS	&	2550	&	QUARTC	&	5000	\\
CRAGGLVY	&	5000	&	EIGENBLS	&	2550	&	SCHMVETT	&	5000	\\
CURLY10	&	10000	&	EIGENCLS	&	2652	&	SINQUAD	&	5000	\\
CURLY20	&	10000	&	ENGVAL1	&	5000	&	SINQUAD2	&	5000	\\
CURLY30	&	10000	&	EXTROSNB	&	1000	&	SPARSINE	&	5000	\\
DIXMAANA1	&	3000	&	FLETBV3M	&	5000	&	SPARSQUR	&	10000	\\
DIXMAANB	&	3000	&	FLETCBV2	&	5000	&	SPINLS	&	1327	\\
DIXMAANC	&	3000	&	FLETCHCR	&	1000	&	SPMSRTLS	&	4999	\\
DIXMAAND	&	3000	&	FMINSRF2	&	5625	&	SROSENBR	&	5000	\\
DIXMAANE1	&	3000	&	FMINSURF	&	5625	&	SSBRYBND	&	5000	\\
DIXMAANF	&	3000	&	FREUROTH	&	5000	&	SSCOSINE	&	5000	\\
DIXMAANG	&	3000	&	GENHUMPS	&	5000	&	TOINTGSS	&	5000	\\
DIXMAANH	&	3000	&	JIMACK	&	3549	&	TQUARTIC	&	5000	\\
DIXMAANI1	&	3000	&	KSSLS	&	1000	&	TRIDIA	&	5000	\\
DIXMAANJ	&	3000	&	LIARWHD	&	5000	&	VAREIGVL	&	5000	\\
DIXMAANK	&	3000	&	MOREBV	&	5000	&	WOODS	&	4000	\\
DIXMAANL	&	3000	&	MSQRTALS	&	1024	&		&		
	\\\hline
    \end{tabular}
\end{table}





As stopping condition for our algorithm, as well as all other methods considered in the experimentation, we required  $\|\nabla f(x_k)\|_\infty\le 10^{-6}$;  we also set a maximum number of iterations to 1000000, so that we consider a failure each run that ends by hitting this threshold.

As performance metrics, we considered both the runtime and the number of iterations.

The results of our experiments are shown in Figure \ref{fig:pp_hager}. 
As we can see, \texttt{cg\_descent} procedure appears by far superior than all the other solvers on the considered benchmark of problems. On the other hand, GMM outperforms both L-BFGS and scipy's CG method.

The superior depth and quality of implementation of the \texttt{cg\_descent} procedure w.r.t.\ the simple python code implementing the GMM is probably a component that makes the gap so particularly pronounced; on the other hand, although clearly outperformed, GMM still retains a certain degree of comparability with the state-of-the-art method. This consideration makes us believe that the integration of preconditioning or subspace minimization steps in GMM might lead to an easily accessible Python code with a comparable efficiency w.r.t.\ \texttt{cg\_descent}.

\color{black}

\begin{figure}[htbp]
    \centering
    \begin{minipage}{0.45\textwidth}
        \centering
        \includegraphics[width=\linewidth]{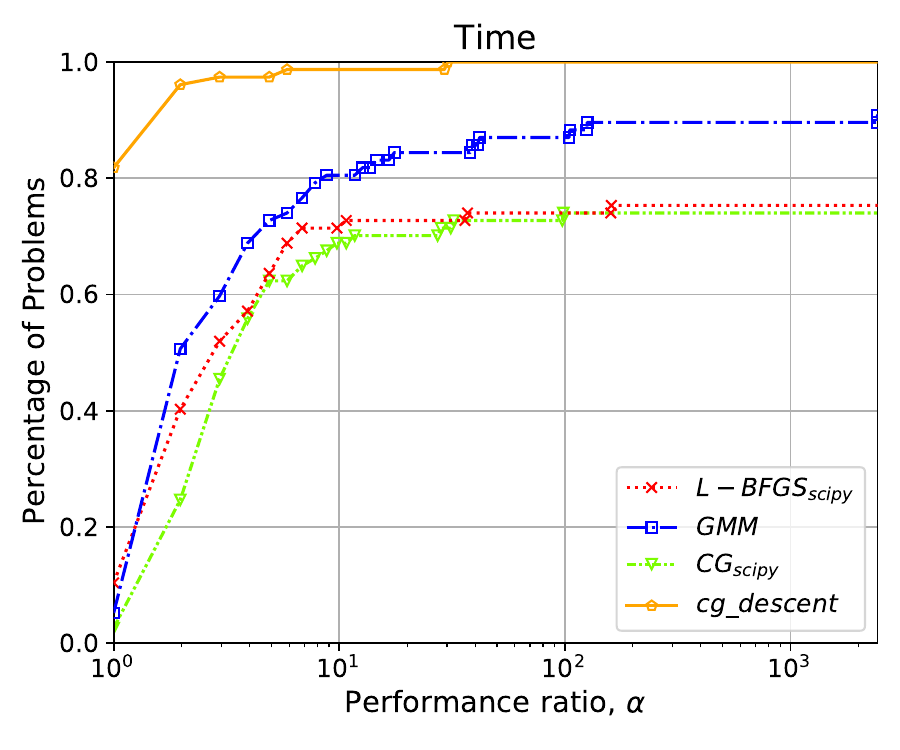}
    \end{minipage}
    \hfil
    \begin{minipage}{0.45\textwidth}
        \centering
        \includegraphics[width=\linewidth]{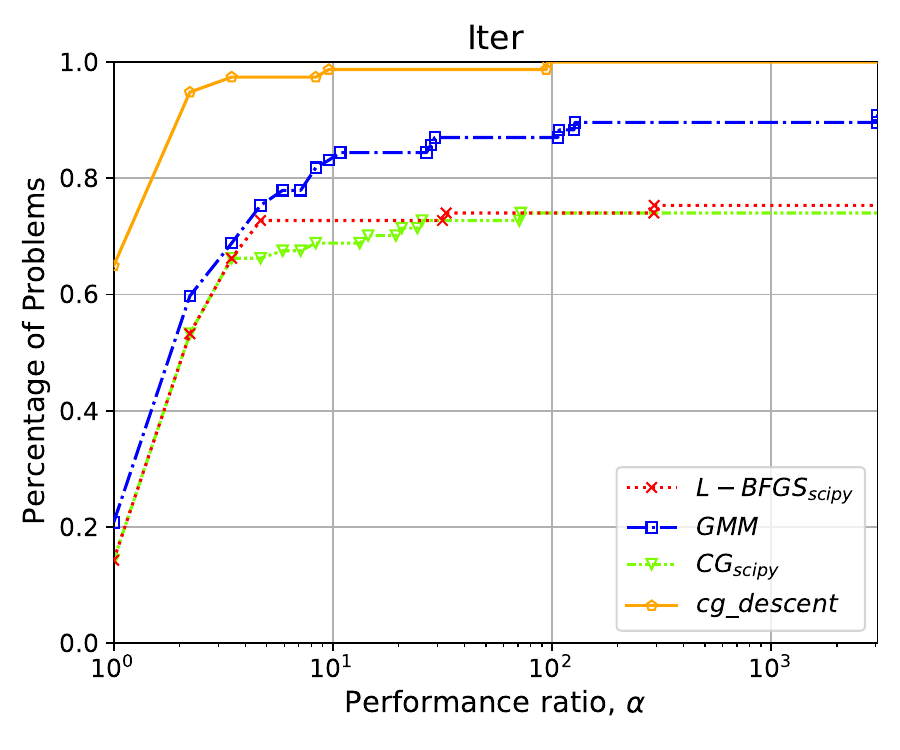}
    \end{minipage}
    \caption{Performance profiles for the comparison between GMM, \texttt{cg\_descent} and the Scipy's implementations of the conjugate gradient method and L-BFGS, on the 77 problems considered.}
    \label{fig:pp_hager}
\end{figure}

\section{Conclusions}
\label{sec:conclusions}
In this work, we introduced a general framework of gradient methods with momentum for nonconvex optimization. For the proposed class of algorithms, we proved global convergence and optimal worst-case complexity bound. The assumptions required to obtain the theoretical results are lighter to check, from a computational perspective, than those required in related works from the literature.  This result allowed to devise particularly efficient ways of implementing the proposed method. By thorough computational results we showed that a rather simple implementation of the novel algorithm outperforms simpler versions of conjugate gradient methods, is competitive with L-BFGS, and is somewhat comparable to state-of-the-art solvers for nonlinear optimization problems, especially in particular classes of problems, such as multi-layer perceptron training tasks. 

These positive preliminary feedback on our procedure motivates future studies aimed at the further improvement of the performance of our procedure. In particular, the integration of some limited-memory and subspace optimization subroutines in \texttt{cg\_descent} within the GMM framework might be subject of future studies. 
\color{black}


\section*{Declarations}

\subsection*{Funding}

No funding was received for conducting this study.

\subsection*{Competing interests}

The authors have no competing interests to declare that are relevant to the content of this article.

\subsection*{Data Availability Statement}
Data sharing is not applicable to this article as no new data were created or analyzed in this study.

\subsection*{Code Availability Statement}
The code developed for the experimental part of this paper is publicly available at \href{https://github.com/gliuzzi/GMM}{\tt https://github.com/gliuzzi/GMM}.

\subsection*{Acknowledgments}
The authors are very grateful to Prof.\ Luigi Grippo for the useful discussions at the early stages of this work. We are also extremely thankful to the Editor-in-Chief and the anonymous referees whose very stimulating comments allowed us to improve the soundness of our work.


\end{document}